\documentclass[12pt]{amsart}
\usepackage{amsmath,amsfonts,amscd,bezier,amsthm,amssymb}
\usepackage{color}
\usepackage{cite}
\usepackage{parskip}
\usepackage{graphicx}
\newtheorem{theorem}{Theorem}[section]
\newtheorem{lemma}[theorem]{Lemma}

\theoremstyle{definition}

\newtheorem{remark}[theorem]{Remark}

\newtheorem{proposition}[theorem]{Proposition}

\newcommand{\R}{\mathbb{{R}}}
\newcommand{\Z}{\mathbb{{Z}}}
\newcommand{\N}{\mathbb{{N}}}

\newcommand{\C}{\mathbb{{C}}}

\numberwithin{equation}{section}

\allowdisplaybreaks

\makeatletter
\DeclareRobustCommand{\eqrefp}[2]{%
  \textup{\tagform@{\refp{#1}{#2}}}}
\makeatother
\DeclareRobustCommand{\refp}[2]{%
  \expandafter\ifx\csname r@#1\endcsname\relax
    \textbf{??}%
  \else
    \edef\areferencia{\ref{#1}-)}%
    \expandafter\eqrefpaux\areferencia-#2%
  \fi}
\def\eqrefpaux#1-#2){#1}

\begin{document}
\title[Asymptotic behavior for the discrete in time heat equation]{Asymptotic behavior for the discrete in time heat equation}

\author[L. Abadias]{Luciano Abadias}
\address[L. Abadias]{Departamento de Matem\'aticas, Instituto Universitario de Matem\'aticas y Aplicaciones, Universidad de Zaragoza, 50009 Zaragoza, Spain.}
\email{labadias@unizar.es}

\author[E. Alvarez]{Edgardo Alvarez}
\address[E. Alvarez]{Universidad del Norte, Departamento de Matem\'aticas y Estad\'istica, Barranquilla, Colombia}
\email{ealvareze@uninorte.edu.co}


\thanks{The first named author has been partly supported by Project MTM2016-77710-P, DGI-FEDER, of the MCYTS and Project E26-17R, D.G. Arag\'on, Universidad de Zaragoza, Spain.}



\date{}
\maketitle

\begin{abstract}
In this paper we investigate the asymptotic behavior and decay of the solution of the discrete in time $N$-dimensional heat equation. We give a convergence rate with which the solution tends to the discrete fundamental solution, and the asymptotic decay, both in $L^p(\R^N).$ Furthermore we prove optimal $L^2$-decay of solutions. Since the technique of energy methods is not applicable, we follow the approach of  estimates based on the discrete fundamental solution which is given by an original integral representation and also by MacDonald's special functions. As a consequence, the analysis is different to the continuous in time heat equation and the calculations are rather involved.
\end{abstract}

\section{Introduction}

The linear heat equation $u_t=\Delta u$ is one of the most studied problems in the theory of partial differential equations. It was introduced by J. Fourier (see \cite{Fourier}) to model several diffusion phenomena. Since then, it has been applied in the study of different processes in many mathematical areas such as PDEs, functional analysis, harmonic analysis, probability, among others. The nature of this problem is well known and we will not further explain it.

One of the aspects of interest, see \cite{F,V,Z}, is the large time behavior of solutions of the heat problem
\begin{equation}\label{eq0.1}
\left\{
\begin{array}{lll}
u_t= \Delta u,\quad t\geq 0,x\in\R^N, \\
u(0) =f(x).
\end{array}
\right.
\end{equation}
If $f\in L^1(\R^N),$ the solution of \eqref{eq0.1} on $L^p(\R^N)$ is  $u(t)=G_t*f,$ where $*$ denotes the classical convolution on $\R^N$ and $$G_t(x)=\frac{1}{(4\pi t)^{N/2}}e^{-\frac{|x|^2}{4t}},\quad t>0, x\in\R^N,$$ is the heat kernel. It is known that integrating over all of $\R^N$ we get that the total mass of solutions is conserved for all time, that is, $$\int_{\R^N}u(t,x)\,dx=\int_{\R^N}u(0,x)\,dx.$$ This fact leads us to think that the total mass of solutions should have importance in the asymptotic behavior of solutions. Indeed, it is well known that if $M=\int_{\R^N}u(0,x)\,dx$ then \begin{equation}\label{0.1}t^{\frac{N}{2}(1-1/p)}\|u(t)-MG_{t}\|_{p}\to0,\quad\text{as }t\to \infty,\end{equation} for $1\leq p\leq \infty.$ The previous estimate shows that the difference on $L^p(\R^N)$ between the solution $u(t)$ and
$MG_t$ decays to zero like $o(\frac{1}{t^{\frac{N}{2}(1-1/p)}})$ as $t$ goes to infinity.

Also, it is known that the $p$-norms of the solution vanish as $t\to\infty$ for $p>1.$ This fact is known as that the $p$-energy is not conservative. More precisely $$\|u(t)\|_{p}\lesssim \frac{\|f\|_{q}}{t^{\frac{N}{2}(1/q-1/p)}},\quad \|\nabla u(t)\|_{p}\lesssim \frac{\|f\|_{q}}{t^{\frac{N}{2}(1/q-1/p)+1/2}},\quad \|\Delta u(t)\|_{p}\lesssim \frac{\|f\|_{q}}{t^{\frac{N}{2}(1/q-1/p)+1}},$$
for $f\in L^q(\R^N),1\leq q\leq p\leq\infty.$

One can consider the first moment as the vector quantity $\int_{\R^n}x\,u(t,x)\,dx.$ It can be seen that such moment is also conserved in time for the solution of \eqref{eq0.1} whenever $(1+|x|)f\in L^1(\R^N).$ Moreover, under such assumption we are able to improve the convergence \eqref{0.1}, that is,$$t^{\frac{N}{2}(1-1/p)+1/2}\|u(t)-MG_{t}\|_{p}\to0,\quad\text{as }t\to \infty.$$ However, the second moment $$\int_{\R^n}|x|^2\,u(t,x)\,dx=\int_{\R^n}|x|^2\,f(x)\,dx+2Nt$$ is not conservative. In fact, it is known that only integral quantities conserved by the solutions of \eqref{eq0.1} are the mass and the first moment.

This type of large-time asymptotic results have been also studied for several diffusion problems. For example in \cite{DD,ZE,GV,KS,N} the authors studied large-time behaviour and other asymptotic estimates for the solutions of different diffusion problems in $\R^N,$ and similar aspects are studied for open bounded domains in \cite{D2, GV}. Estimates for heat kernels on manifolds have been also studied in \cite{Gr,L,D}. In \cite{M}, the author obtains gaussian upper estimates for the heat kernel associated to the sub-laplacian on a Lie group, and also for its first-order time and space derivatives.

On the other hand, finite differences, sometimes also called discrete derivatives, were introduced some centuries ago, and they have been used along the literature in different mathematical problems, mainly in approximation of derivatives for the numerical solution of differential equations and partial differential equations. The most knowing ones are the forward, backward and central differences (the forward and backward differences are associated to the Euler, explicit and implicit, numerical methods). We denote them in the following way; let $h>0,$ for a function $u$ defined on the mesh $\Z_{h}:=\{nh\,:\, n\in \Z\}$ we write $$\delta_{\text{right}}u(nh):=\frac{u((n+1)h)-u(nh)}{h},\quad \delta_{\text{left}}u(nh):=\frac{u(nh)-u((n-1)h)}{h},$$and $$\delta_c u(nh):=\biggl( \frac{\delta_{\text{right}}+\delta_{\text{left}}}{2}\biggr)u(nh)=\frac{u((n+1)h)-u((n-1)h)}{2h}.$$

In the last years, and taking as a guide the paper \cite{Bateman}, several authors have been working in the context of partial difference-differential equations (\cite{ADT, AL, C, C2, PAMS, LR}) from an specific point of view; in that papers the approach has been focused in mathematical analysis, more precisely, harmonic analysis, functional analysis and fractional differences. Particularly in \cite{ADT} it is shown that the operators $\delta_{\text{right}}$ and $\delta_{\text{left}}$ generate markovian $C_0$-semigroups on $\ell^p(\Z).$ Also, in \cite{C}, the authors study harmonic properties of the solution of the heat problem on one-dimensional graphs (the mesh $\Z_h$), and the wave equation on graphs is studied in \cite{LR}. An abstract approach for discrete in time Cauchy problems is given in \cite{PAMS}. Also, non-local problems in the discrete framework appear in \cite{AL, C2}.

The previous comments motivate the main aim of this paper; let $h>0,$ we consider the first order Cauchy problem for the heat equation in discrete time
\begin{equation}\label{eq1.1}
\left\{
\begin{array}{lll}
\delta_{\text{left}}u(nh,x) = \Delta u(nh,x)+g(nh,x),\quad n\in\N,\,x\in \R^N, \vspace{2mm}\\
u(0,x) =f(x),
\end{array}
\right.
\end{equation}
where $\Delta$ is the classical laplacian on $L^p(\R^N)$, $u$ is defined on  $\N_0^h\times \R^N$, with $\N_0^h:=\{nh:n\in\N_0\}$, $f$ is defined on $\R$ and $g$ is defined on $\N^h\times \R^N,$ with $\N^h:=\{nh:n\in \N\}.$

Along the paper we study asymptotic behavior and decay of the solution of \eqref{eq1.1}. For that purpose, we need to know properties of the fundamental solution of the homogeneous problem associated to \eqref{eq1.1} (when $g=0$). In fact, one of the key points to get such asymptotic properties is an integral representation of the fundamental solution for the associated homogeneous equation. Furthermore, we describe explicitly this solution in terms of MacDonald's functions which arise naturally from the integral representation of the solution. This representation is quite original and allows to study the decay of solutions for the problem \eqref{eq1.1} when the initial datum belongs to $p$-integrable Lebesgue spaces. Moreover, both the integral representation and the explicit expression via MacDonald's functions allow to give a quantitative rate at which the solution converges to M times the fundamental solution, where $M$ will denote, as in the continuous case, the initial mass of solution. The techniques used to obtain our results differs to the continuous case because we have to deal with the integral representation and
asymptotic properties of MacDonald's special functions. We also note to the reader that doing the relation $t=nh,$ the asymptotics of $G_t$ will be similar to $\mathcal{G}_{n,h}$ as $t\to\infty$ or equivalently $n\to\infty,$ where $\mathcal{G}_{n,h}$ will denote the fundamental solution of the homogeneous problem associated to \eqref{eq1.1}.

In this paper we are not interested in the study of the convergence as $h\to 0$ of solutions of the problems \eqref{eq1.1} (depending on $h$) to the classical heat problem. However, it can be seen as a natural problem studied in semigroup theory via Yosida approximants (see Remark \ref{remark2.1}). Also, one can think about the possibility to consider similar problems to \eqref{eq1.1} but considering the discrete derivatives $\delta_{\text{right}}$ or $\delta_c.$ However, as we explain in Remark \ref{remark2.2}, the fundamental solutions to that problem does not have good properties.

The paper is organized as follows. Section \ref{homogeneo} is focused in the fundamental solution of the homogeneous problem associated to \eqref{eq1.1}. We introduce an integral representation and the explicit expression via MacDonald's functions. We deduce basic properties, we calculate its gradient and laplacian, and we see that the mass and the first moment of solutions of the homogeneous problem are conservative in discrete time $nh,$ and not the second moment. Also some pictures of the continuous and discrete gaussian kernels, with their corresponding comments, are stated. In Section \ref{3} we give pointwise and $L^p$ asymptotic upper bounds for the fundamental solution $\mathcal{G}_{n,h},$ and we use such estimates to prove in Section \ref{4} that the $p$-energies of solutions of \eqref{eq1.1} are dissipative. Section \ref{5} is the main part of the paper; we prove the asymptotic behaviour for the discrete in time heat problem (Theorem \ref{Theorem5.1}). In Section 6 we success in proving optimal $L^2$-decay estimates for the solution of the homogeneous problem associated to \eqref{eq1.1}. The proof is based on Fourier analysis techniques. Finally we include an Appendix where we show some basic properties of Gamma and MacDonald's functions, and a technical result about integrability.


\section{The discrete gaussian fundamental solution}\label{homogeneo}

In this section we study the fundamental solution for the homogeneous discrete in time heat initial value problem on the Lebesgue $L^p(\R^N)$  spaces.
Let $h>0,$ we consider
\begin{equation}\label{maineq1}
\left\{
\begin{array}{lll}
\delta_{\text{left}}u(nh,x) = \Delta u(nh,x),\quad n\in\N,\,x\in\mathbb{R}^N, \vspace{2mm}\\
u(0,x) =f(x),
\end{array}
\right.
\end{equation}
where $u$ and $f$ are functions defined on $\N_0^h\times \R^N$ and $\R^N$, respectively. Formally, one can write the solution in the following way
\begin{equation*}\label{eq-1}
u(nh,x) =\frac{1}{h^n} (1/h - \Delta)^{-n} f(x),\quad n\in \N,\,x\in\mathbb{R}^N,
\end{equation*}
whenever the resolvent operator $(1/h - \Delta)^{-1}$ has sense. It is well known that
the laplacian operator $\Delta$ associated to the standard heat equation in continuous time on $L^p(\R^N)$ for $1\leq p\leq \infty$  generates the gaussian semigroup with convolution kernel $$G_t(x)=\frac{1}{(4\pi t)^{N/2}}e^{-\frac{|x|^2}{4t}},\quad t>0,\, x\in\R^N.$$ From semigroup theory (see \cite[Corollary 1.11]{EN}) we obtain
$$
	u(nh,\cdot)=\frac{1}{h^n} (1/h - \Delta)^{-n} f(\cdot)=  \frac{1}{h^n\Gamma(n)} \int_0^\infty e^{-t/h} t^{n-1} (G_t*f)(\cdot)dt:=(\mathcal{G}_{n,h}*f)(\cdot),\,\,f \in L^p(\R^N),
$$
where $*$ denotes the classical convolution on $\R^N$ and  \begin{equation}\label{eq2.3}\mathcal{G}_{n,h}(x)= \frac{1}{h^n\Gamma(n)} \int_0^\infty e^{-t/h} t^{n-1} G_t(x)\,dt,\quad n\in\N,x\in\R^N\setminus\{0\}.
\end{equation}

\begin{remark}\label{remark2.1}{\it
Note that fixed a positive number $t>0,$ the Yosida approximants  $(1-\frac{t}{n}\Delta)^{-n}$ (see \cite[Theorem 3.5]{EN}) allow to approximate the gaussian $C_{0}$-semigroup $G_t$ as $n\to\infty.$ Writing $h=t/n,$ the previous convergence shows that the gaussian semigroup can be approximated  by the solutions of the discrete in time problems \eqref{maineq1} as the mesh $h\to 0.$}
\end{remark}

\begin{remark}\label{remark2.2}{\it
It is easy to see that if we consider the forward difference $\delta_{\text{right}}$ on \eqref{maineq1}, then formally the solution of the problem would be $u(nh,\cdot) =\frac{1}{h^n} (1/h + \Delta)^{n} f(\cdot),$ which is not defined (bounded) on $L^p(\R^N).$

Also, for the central difference $\delta_{\text{c}}$, the fundamental solution would be given by $$\int_{0}^{\infty}J_n(t/h)G_{t}(x)\,dt,$$ where $J_n$ are the Bessel functions of first kind. In this case is not difficult to prove that the solution is bounded on $L^p(\R^N),$ however it does not have as good properties as $\mathcal{G}_{n,h}$ satisfies, for example the contractivity on $L^1(\R^N).$

These are the main reasons because of we consider the discrete in time heat problem with the backward difference $\delta_{\text{left}}.$}
\end{remark}

Now we will see the explicit expression of the fundamental solution $\mathcal{G}_{n,h}$ in terms of special functions. By \cite[p.363 (9)]{G} we have
\begin{eqnarray}\label{Macdonald}
	\nonumber \mathcal{G}_{n,h}(x)&=& \frac{1}{h^n\Gamma(n)(4\pi)^{N/2}} \int_0^\infty e^{-(t/h+|x|^2/4t)}t^{n-N/2-1}\,dt\\
	& = & \frac{2}{\Gamma(n)(4\pi h)^{N/2}}\biggl( \frac{|x|}{2\sqrt{h}} \biggr)^{n-N/2}K_{n-N/2}\biggl(  \frac{|x|}{\sqrt{h}} \biggr), \quad  n\in\N,x \in \R^N\setminus\{0\}.
\end{eqnarray}
Here, the functions $K_{\nu}$ denote  the  Bessel functions of imaginary argument, also called MacDonald's functions or modified cylinder functions (see Section \ref{Appendix}). Observe that the identity has not pointwise sense for $x=0$ if $ N/2-n\geq 0.$ In fact, for that values $ N/2-n\geq 0$, taking $|x|\to 0$ in \eqref{Macdonald} and using (P4) and (P6) of Appendix one gets $\mathcal{G}_{n,h}(x)\to\infty.$ However, as we will see, good properties on $L^p(\R^N)$ hold. For the case $ N/2-n< 0,$ by (P4) we have $\mathcal{G}_{n,h}(x)\to \frac{\Gamma(n-N/2)}{\Gamma(n)(4\pi h)^{N/2}}$ as $|x|\to 0.$

\begin{remark}{\it
The gaussian kernel satisfies the semigroup property on time, $G_t\ast G_s=G_{t+s}$. Since $\mathcal{G}_{n,h}$ is given by natural powers of the resolvent operator of the laplacian, it satisfies the discrete semigroup property. Indeed, we also can prove that property using the expression \eqref{eq2.3} as follows,
\begin{align*}
(\mathcal{G}_{n,h}\ast \mathcal{G}_{m,h})(x)&=\frac{1}{h^{n+m}\Gamma(n)\Gamma(m)}\int_{\mathbb{R}^N}\left(\int_0^{\infty}
\int_0^{\infty}e^{-\frac{t+s}{h}}t^{n-1}s^{m-1}G_t(x-y)G_s(y)\,ds\,dt\right)dy\\
&=\frac{1}{h^{n+m}\Gamma(n)\Gamma(m)}\int_0^{\infty}
\left(\int_0^{\infty}e^{-\frac{t+s}{h}}t^{n-1}s^{m-1}G_{t+s}(y)\,ds\right)dt\\
&=\frac{1}{h^{n+m}\Gamma(n)\Gamma(m)}\int_0^{\infty}
\left(\int_t^{\infty}e^{-\frac{\sigma}{h}}t^{n-1}(\sigma-t)^{m-1}G_{\sigma}(x)\,d\sigma\right)dt\\
&=\frac{1}{h^{n+m}\Gamma(n)\Gamma(m)}\int_0^{\infty}e^{-\frac{\sigma}{h}}G_{\sigma}(x)
\left(\int_0^{\sigma}t^{n-1}(\sigma-t)^{m-1}\,dt\right)d\sigma\\
&=\frac{1}{h^{n+m}\Gamma(n)\Gamma(m)}\int_0^{\infty}e^{-\frac{\sigma}{h}}G_{\sigma}(x)
\sigma^{m+n-1}B(n,m)\,d\sigma=\mathcal{G}_{n+m,h}(x).
\end{align*}
Here $B(n,m)$ is the Beta function.}
\end{remark}

In the following we denote
\begin{equation*}\label{poissonkernel}
p_{n,h}(t)=:\frac{1}{h^n\Gamma(n)}e^{-t/h} t^{n-1},\quad n\in\N.
\end{equation*}
Then we can write
	\begin{equation}\label{Integralrepresentationheat}
	\mathcal{G}_{n,h}(x) =\int_0^{\infty}p_{n,h}(t)G_t(x)dt,\quad x\neq 0.
	\end{equation}
The above integral representation is a discretization formula for the gaussian semigroup. The case $h=1$ was treated in \cite{PAMS} for a general $C_0$-semigroup on an abstract context.



Next, we refer to the function $\mathcal{G}_{n,h}$ as the \emph{fundamental solution} for the problem \eqref{maineq1}. The following proposition states some basic properties of it.

\begin{proposition}\label{HeatProperties}
The function $\mathcal{G}_{n,h}$ satisfies:
	\begin{enumerate}
	\item[(i)] $\displaystyle\mathcal{G}_{n,h}(x)>0,\quad n\in\N,x \neq 0.$
	\item[(ii)] $\displaystyle \int_{\R^N}\mathcal{G}_{n,h}(x)\ dx=1$.
	\item[(iii)] $\displaystyle\mathcal{F}({\mathcal{G}}_{n,h})(\xi)=\frac{1}{(1+h|\xi|^2)^n},\quad \xi\in\R^N.$
	\item[(iv)] $\dfrac{\mathcal{G}_{n,h}(x)-\mathcal{G}_{n-1,h}(x)}{h} = \Delta \mathcal{G}_{n,h}(x),\quad n\geq 2,x\neq 0.$
    \item[(v)] $\displaystyle \int_{\R^N}|x|^2\mathcal{G}_{n,h}(x)\,dx=2Nnh.$

	\end{enumerate}
\end{proposition}
\begin{proof} (i) It is clear by \eqref{Integralrepresentationheat}. (ii) Note that $\int_0^{\infty}p_{n,h}(t)dt=1$ and  $\int_{\R^N}G_t(x)dx=1,$ then the result follows from the Fubini's theorem. (iii) It is known that $\mathcal{F}({G}_{t})(\xi)=e^{-t|\xi|^2}$, for $\xi\in\R^N,$ then  by \eqref{Integralrepresentationheat} one gets $$\mathcal{F}(\mathcal{G}_{n,h})(\xi)=\int_{0}^{\infty}p_{n,h}(t)\mathcal{F}({G}_t)(\xi)\,dt =\frac{1}{(1+h|\xi|^2)^n}.$$
\noindent (iv) First of all, observe that $\frac{d}{dt}p_{n,h}(t)=-\frac{1}{h}(p_{n,h}(t)-p_{n-1,h}(t))$ for $n\geq 2.$ Then integrating by parts we get

\begin{eqnarray*}\frac{\mathcal{G}_{n,h}(x)-\mathcal{G}_{n-1,h}(x)}{h}&=&\int_{0}^{\infty}p_{n,h}(t)\frac{\partial}{\partial t}G_t(x)\ dt\\
	&=&\int_{0}^{\infty}p_{n,h}(t)\Delta G_t(x)\ dt\\
	&= &\Delta \mathcal{G}_{n,h}(x), \quad x\neq 0,
	\end{eqnarray*}
	where we have used that $\lim_{t \to 0^+}$ and $\lim_{t\to\infty}$ of $p_{n}(t)G_t(x)$ vanishes. 
(v) It follows easily by the second moment of $G_t(x)$ and the representation \eqref{eq2.3}.
\end{proof}

\begin{remark}\label{norm1}{\it
Observe that one can prove the above properties via the expression \eqref{Macdonald} given by the MacDonald's function. For example,  from (P1) of Appendix we get the positivity of the fundamental solution. Furthermore,  by \cite[p. 668 (16)]{G}  it follows
\begin{align*}
\int_{\R^N}\mathcal{G}_{n,h}(x)\,dx&= \frac{2}{\Gamma(n)(4\pi h)^{1/2}}\int_{\mathbb{R}^N}\biggl( \frac{|x|}{2\sqrt{h}} \biggr)^{n-N/2}K_{n-N/2}\biggl(  \frac{|x|}{\sqrt{h}} \biggr)\, dx\\
&=\frac{2^{1+\frac{N}{2}-n}}{\Gamma(n)(4\pi h)^{1/2}h^{\frac{n}{2}-\frac{N}{4}}}\left(\frac{N\pi^{N/2}}{\Gamma(\frac{N}{2}+1)}
\int_0^{\infty}r^{n-N/2}K_{n-N/2}\biggl(  \frac{r}{\sqrt{h}} \biggr)\, dr\right)\\
&=\frac{\frac{N}{2}\Gamma(\frac{N}{2})h^{\frac{N}{2}+\frac{N}{4}}}{h^{\frac{N}{2}+\frac{N}{4}}\Gamma(\frac{N}{2}+1)}=1.
\end{align*}	
Also note that by $\frac{\partial |x|}{\partial x_j}=\frac{x_j}{|x|}$ and (P2) of Appendix, we obtain \begin{equation}\label{derivative}
\frac{\partial}{\partial x_j}\mathcal{G}_{n,h}(x)=\frac{-2}{\Gamma(n)(4\pi h)^{N/2}\sqrt{h}}\frac{x_j}{|x|}\biggl( \frac{|x|}{2\sqrt{h}} \biggr)^{n-N/2}K_{n-N/2-1}\biggl(  \frac{|x|}{\sqrt{h}} \biggr),\end{equation}
 and then derivating once more in the previous expression and taking into account (P3) and (P7) (with $\nu=n-\frac{N}{2}-1$) of Appendix, we have
\begin{align}\label{derivative}
\nonumber\frac{\partial^2}{\partial x_j^2}\mathcal{G}_{n,h}(x)&=-\frac{1}{h\Gamma(n)(4\pi h)^{N/2}}
\left(\frac{|x|}{2\sqrt{h}}\right)^{n-\frac{N}{2}-1}K_{n-\frac{N}{2}-1}\left(\frac{|x|}{\sqrt{h}}\right)\\
\nonumber &+\frac{x_j^2}{h}\frac{1}{2h\Gamma(n)(4\pi h)^{N/2}}
\left(\frac{|x|}{2\sqrt{h}}\right)^{n-\frac{N}{2}-2}K_{n-\frac{N}{2}}\left(\frac{|x|}{\sqrt{h}}\right)\\
\nonumber &+N\frac{x_j^2}{h}\frac{1}{4h\Gamma(n)(4\pi h)^{N/2}}
\left(\frac{|x|}{2\sqrt{h}}\right)^{n-\frac{N}{2}-3}K_{n-\frac{N}{2}-1}\left(\frac{|x|}{\sqrt{h}}\right)\\
\nonumber &-2(n-1)\frac{x_j^2}{h}\frac{1}{4h\Gamma(n)(4\pi h)^{N/2}}
\left(\frac{|x|}{2\sqrt{h}}\right)^{n-\frac{N}{2}-3}K_{n-\frac{N}{2}-1}\left(\frac{|x|}{\sqrt{h}}\right).
\end{align}
Now, since $\sum_{j=1}^N\frac{x_j^2}{h}=\left(\frac{|x|}{\sqrt{h}}\right)^2$, we get
$$
\Delta\mathcal{G}_{n,h}(x)=\frac{\mathcal{G}_{n,h}(x)-\mathcal{G}_{n-1,h}(x)}{h}.
$$

Finally observe that the mean square displacement can be also calculated in the following way; using \eqref{Macdonald}, a change of variables and \cite[p.668 (16)]{G}, we have
\begin{align*}
\int_{\R^N}|x|^2\mathcal{G}_{n,h}(x)\,dx&=\frac{2}{\Gamma(n)(4\pi h)^{N/2}(2\sqrt{h})^{N-1/2}}\int_{\R^N} |x|^{n-N/2+2} K_{n-N/2}\biggl(  \frac{|x|}{\sqrt{h}} \biggr)\,dx\\
&=\frac{2}{\Gamma(n)(4\pi h)^{N/2}(2\sqrt{h})^{N-1/2}}\left(\frac{N\pi^{N/2}}{\Gamma(\frac{N}{2}+1)}
\int_0^{\infty}r^{n+\frac{N}{2}+1}K_{n-N/2}\biggl( \frac{r}{\sqrt{h}}\biggr)\,dr\right)\\
&=2Nnh.
\end{align*}
}
\end{remark}

\begin{remark}{\it Note that by Proposition \ref{HeatProperties} (i) we have that the total mass of solution of \eqref{maineq1} is conservative in the discrete time $nh,$ that is, $$\int_{\R^N}u(nh,x)\,dx=\int_{\R^N}f(x)\,dx.$$ Moreover, the first moment is also conservative; if $(1+|x|)f\in L^1(\R^N)$ one gets $$\int_{\R^N}x(u(nh,x)-u((n-1)h,x))\,dx=h\int_{\R^N}x \Delta u(nh,x)\,dx=0,$$ and so $\int_{\R^N}x u(nh,x)\,dx=\int_{\R^N}x f(x)\,dx.$ However, as in the continuous case holds, by Proposition \ref{HeatProperties} (v) it follows that the second moment is $$\int_{\R^n}|x|^2\,u(nh,x)\,dx=\int_{\R^n}|x|^2\,f(x)\,dx+2Nnh.$$}
\end{remark}
To finish this section we show some pictures of the fundamental solution of \eqref{maineq1}. We have used Mathematica to make them. The objective is that the reader visualizes the convergence of $\mathcal{G}_{n,h}$ to $G_t$ as the mesh $h\to 0.$

Figure 1 shows, in the one-dimensional case ($N=1$), the Gauss kernel $G_1$ and the fundamental solutions of the discrete problems for several values of $h.$ As we have mentioned, the Yosida approximants (which are the fundamental solutions) converge to the gaussian kernel as $h\to 0$ writing $t=nh.$ Therefore, for the different values of $h,$ we choose $n$ such that $nh=1.$ For example for $h=1/2$ we have represented the fundamental solution $\mathcal{G}_{2,1/2}.$ Also, observe that for $N=1$ the fundamental solution $\mathcal{G}_{n,h}(x)$ is defined on the whole real line since $n-N/2>$ for all $n\in\N.$ However by \eqref{derivative}, and (P6) and (P4) of Appendix we get $$\mathcal{G}_{1,h}'(x)=C_{h}\frac{x}{|x|^{1/2}}K_{-1/2}(\frac{|x|}{\sqrt{h}})=C_{h}\frac{x}{|x|^{1/2}}K_{1/2}(\frac{|x|}{\sqrt{h}})\sim C_{h}\frac{x}{|x|}, \quad  x\to 0,$$ where $C_h$ is a constant depending on $h.$ This shows that $\mathcal{G}_{1,h}$ is not derivable in $x=0$ (see Figure 1 for $h=1$).


\begin{figure}[h]
\caption{}
\centering
\includegraphics[width=0.57\textwidth]{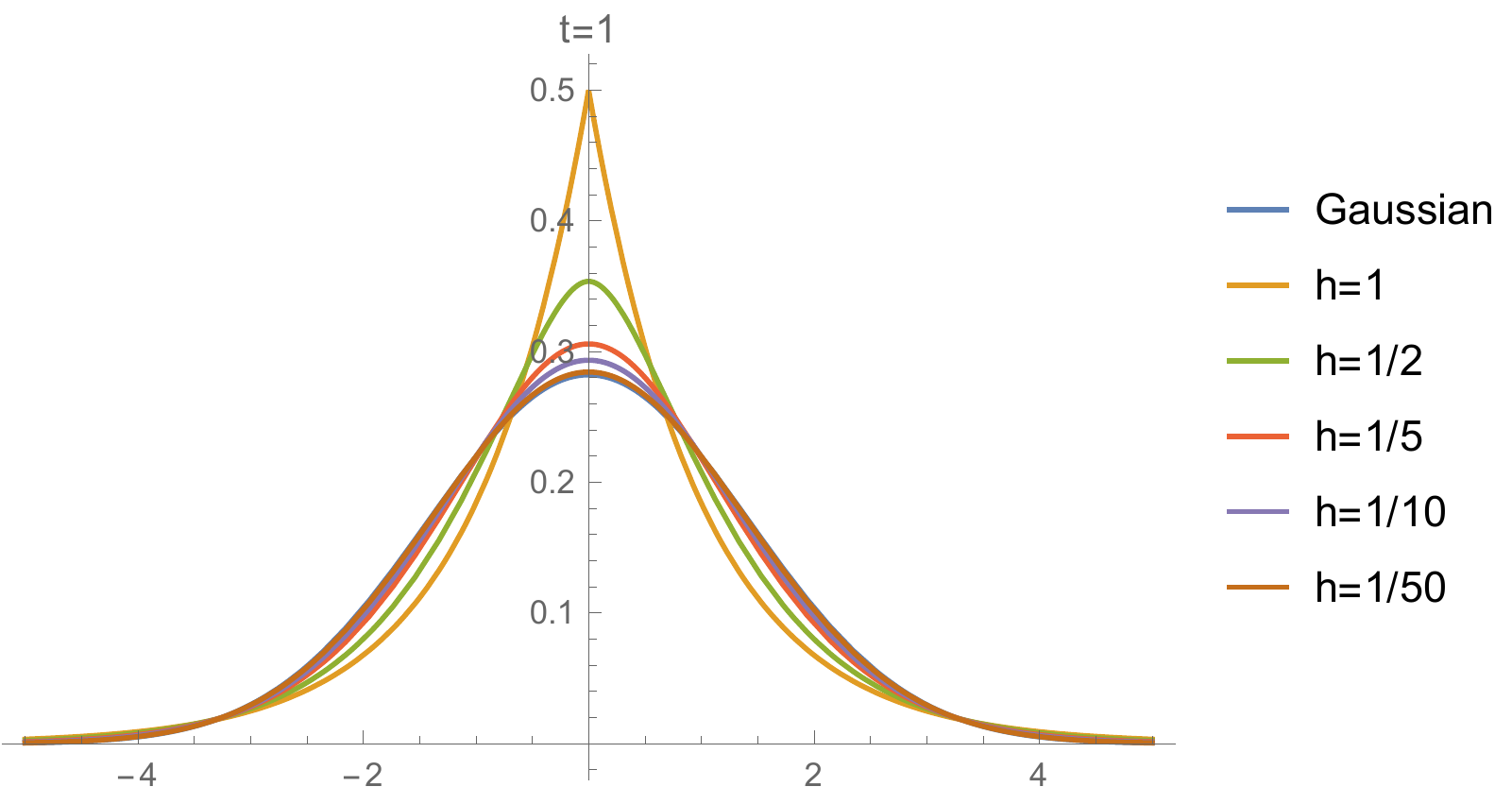}
\end{figure}


Figure 2, Figure 3 and Figure 4 show several approximants to the gaussian $G_1$  in the two-dimensional case ($N=2$). In Figure 2 we observe that $\mathcal{G}_{1,1}(x)\to +\infty$ taking $x\to0,$ as we have commented previously (since $n-N/2=0$ for $n=1,N=2$).


\begin{figure}[h]
\caption{}
\centering
\includegraphics[width=0.57\textwidth]{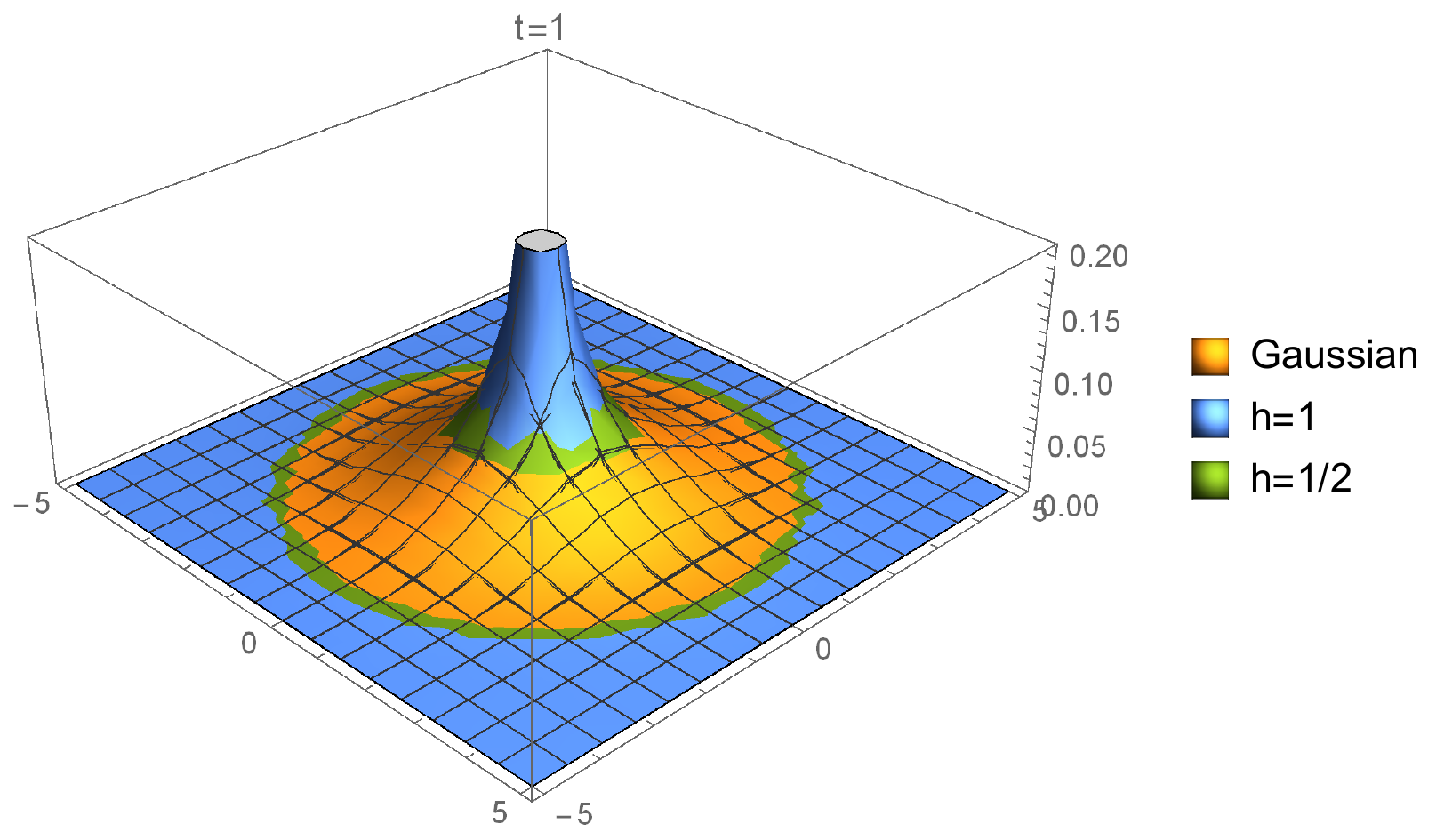}
\end{figure}

\vskip 1.0cm

In Figure 3 we glimpse that  $\mathcal{G}_{2,1/2}(x)$ is well defined for $x=0,$ but it is not differentiable (since $n-N/2=1$). Finally, Figure 4 shows, not only as the approximation improves as $h$ decreases, but also that the function is smoother.


\begin{figure}[h]
\caption{}
\centering
\includegraphics[width=0.57\textwidth]{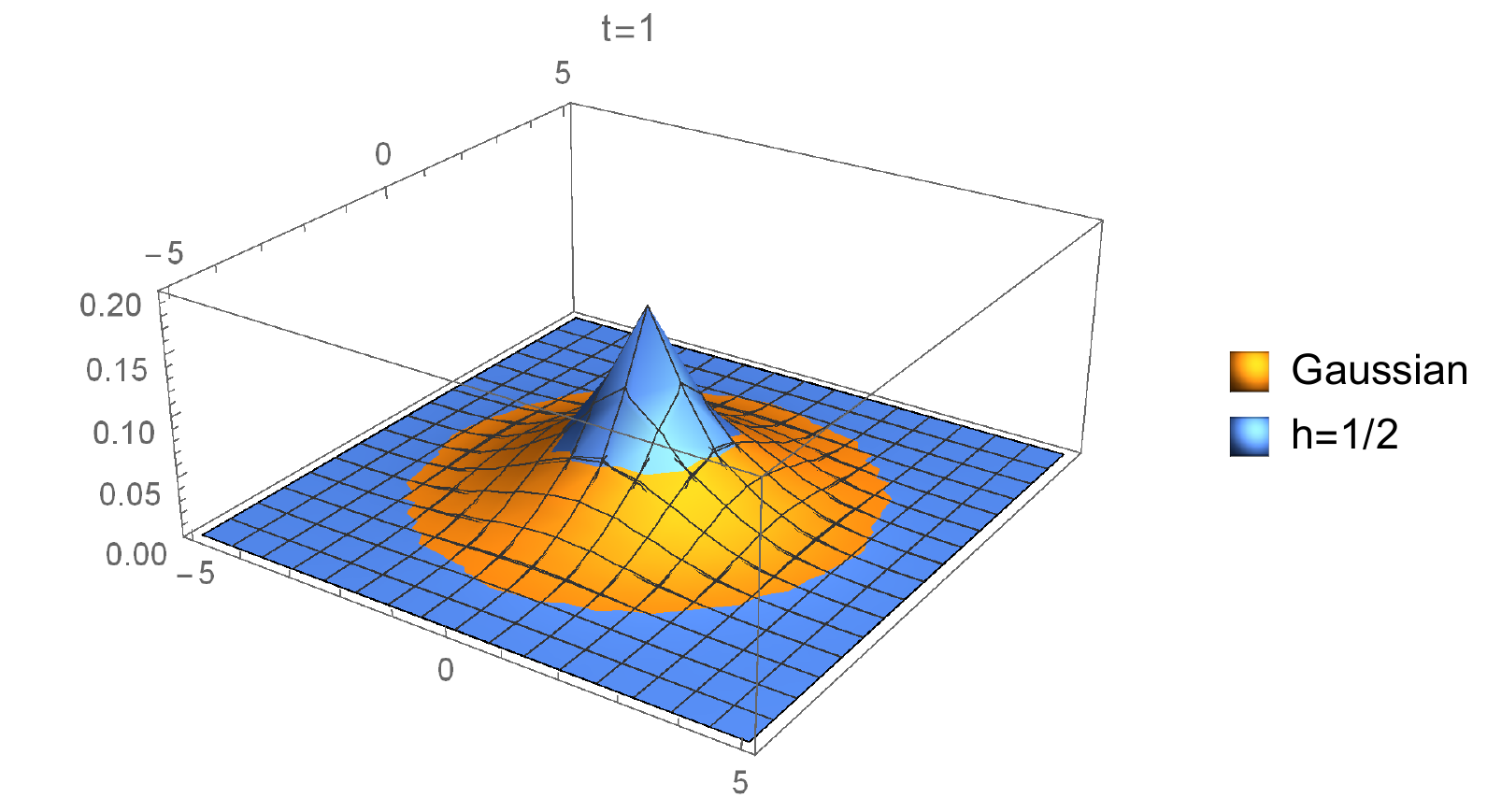}
\end{figure}


\begin{figure}[h]
\caption{}
\centering
\includegraphics[width=0.57\textwidth]{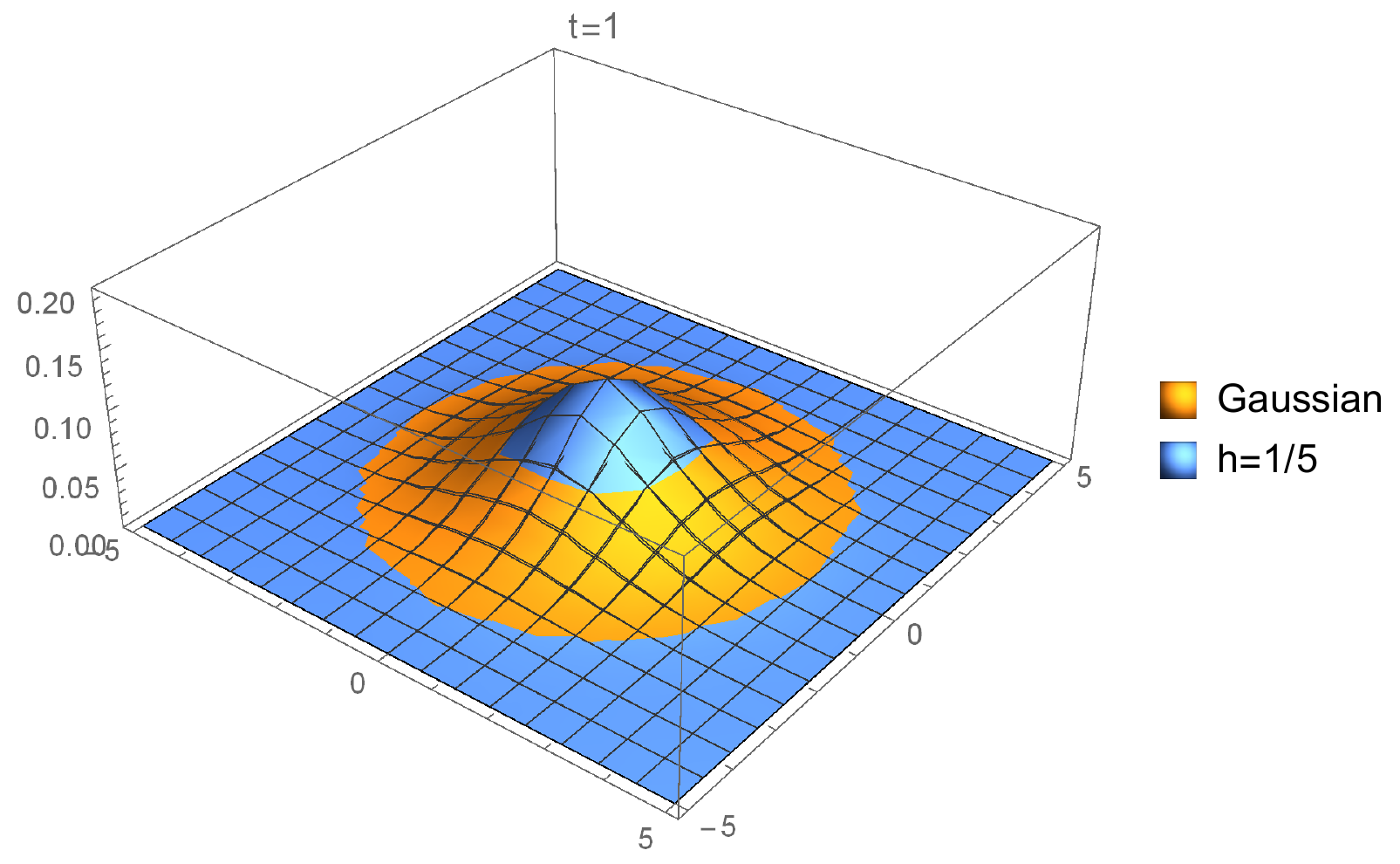}
\end{figure}



\section{Estimates for the fundamental solution}\label{3}

In this section we present pointwise and $p$-norm estimates for the fundamental solution of \eqref{maineq1}.


\begin{theorem}\label{Pointwiseestimates}
Let $R:=\frac{|x|^2}{nh}.$
\begin{itemize}

\item[(i)] If $R\leq 1,$ then  $$|\mathcal{G}_{n,h}(x)|\sim \frac{1}{(nh)^{N/2}},\,\text{ for }n-N/2>0,$$ and if $R\geq 1$, then $$|\mathcal{G}_{n,h}(x)|\lesssim \frac{nh}{|x|^{N+2}}.$$

\item[(ii)] If $R\leq 1,$ then  $$|\nabla\mathcal{G}_{n,h}(x)|\lesssim \frac{|x|}{(nh)^{N/2+1}},\,\text{ for }n-N/2>1,$$ and if $R\geq 1$, then $$|\nabla\mathcal{G}_{n,h}(x)|\lesssim \frac{nh}{|x|^{N+3}}.$$

\item[(iii)]  If $R\leq 1,$ then  $$|\frac{\mathcal{G}_{n,h}(x)-\mathcal{G}_{n-1,h}(x)}{h}|\lesssim \frac{1}{(nh)^{N/2+1}},\,\text{ for }n-N/2>1,$$ and if $R\geq 1$, then $$|\frac{\mathcal{G}_{n,h}(x)-\mathcal{G}_{n-1,h}(x)}{h}|\lesssim \frac{1}{|x|^{N+2}}.$$

\end{itemize}
\end{theorem}
\begin{proof}
(i) Let $R\leq 1.$ By \eqref{Macdonald}, (P4) of Appendix and \eqref{GammaAsymp} we have $$|\mathcal{G}_{n,h}(x)|\sim \frac{\Gamma(n-N/2)}{\Gamma(n)(4\pi h)^{N/2}}\sim\frac{1}{(nh)^{N/2}}.$$ Now let $R\geq 1.$ Along the proof we will use that \begin{equation}\label{exponencial}e^{-\frac{|x|}{\sqrt{h}}}\leq \frac{k!h^{k/2}}{|x|^k},\quad k\in\N_0.\end{equation} By (P5) of Appendix and \eqref{exponencial} for $k=n+N+1$ one gets $$|\mathcal{G}_{n,h}(x)|\lesssim \frac{n^{3N/4+5/4}}{2^n}\frac{nh}{|x|^{N+2}}\lesssim \frac{nh}{|x|^{N+2}}.$$

(ii) Let $R\leq 1.$ Equation \eqref{derivative} implies that \begin{equation}\label{nabla}|\nabla \mathcal{G}_{n,h}(x)|=\frac{2}{\Gamma(n)(4\pi h)^{N/2}\sqrt{h}}\biggl( \frac{|x|}{2\sqrt{h}} \biggr)^{n-N/2}K_{n-N/2-1}\biggl(  \frac{|x|}{\sqrt{h}} \biggr).\end{equation} From (P4) of Appendix and \eqref{GammaAsymp} we have $$|\nabla \mathcal{G}_{n,h}(x)|\sim \frac{\Gamma(n-N/2-1)|x|}{\Gamma(n)h^{N/2+1}}\lesssim\frac{|x|}{(nh)^{N/2+1}}.$$ For $R\geq 1,$ by \eqref{nabla}, (P5) of Appendix and \eqref{exponencial} for $k=n+N+2$ we obtain $$|\nabla\mathcal{G}_{n,h}(x)|\lesssim \frac{hn(n+1)(n+2)(n+3)}{2^n|x|^4}\lesssim \frac{nh}{|x|^4}.$$

(iii) Applying (P3) and (P2) of Appendix in turn, it follows
 \begin{eqnarray*}
\frac{\mathcal{G}_{n,h}(x)-\mathcal{G}_{n-1,h}(x)}{h}&=&\frac{-2|x|^{n-N/2-1}}{\Gamma(n)(4\pi)^{N/2}h^{N/2+1/2}(2\sqrt{h})^{n-N/2}}\\
&\times &\biggl( NK_{n-N/2-1}\biggl(  \frac{|x|}{\sqrt{h}} \biggr)-\frac{|x|}{\sqrt{h}}K_{n-N/2-2}\biggl(  \frac{|x|}{\sqrt{h}} \biggr)\biggr)\\
&:=&(I)+(II).\end{eqnarray*} Let $R\leq 1.$ By (P4) of Appendix and \eqref{GammaAsymp}, that $$|(I)|\lesssim \frac{\Gamma(n-N/2-1)}{\Gamma(n)h^{N/2+1}}\sim \frac{1}{(nh)^{N/2+1}},$$ and $$|(II)|\lesssim \frac{|x|^2}{(nh)^{N/2+2}}\leq \frac{1}{(nh)^{N/2+1}}.$$ Therefore we conclude the result for $R\leq 1.$

Now let $R\geq 1.$ Note that from the part (i), we have $$|\mathcal{G}_{n,h}(x)|, |\mathcal{G}_{n-1,h}(x)|\lesssim \frac{hn^{3N/4+9/4}}{2^{n}|x|^{N+2}}.$$ Then $$|\frac{\mathcal{G}_{n,h}(x)-\mathcal{G}_{n-1,h}(x)}{h}|\lesssim \frac{1}{|x|^{N+2}}.$$

\end{proof}

Now we present the asymptotic decay of the fundamental solution $\mathcal{G}_{n,h}$ in Lebesgue and Sobolev spaces.
\begin{theorem}\label{Gaussianestimates}
Let $1 \leq p \leq \infty,$ then \begin{itemize}

\item[(i)]  $
\|\mathcal{G}_{n,h} \|_{p} \leq C_p\dfrac{1}{(nh)^{\frac{N}{2}(1-1/p)}},\quad n-\frac{N}{2}(1-1/p)>0.
$

\item[(ii)]  $
\|\nabla \mathcal{G}_{n,h}\|_{p} \leq C_p\dfrac{1}{(nh)^{\frac{N}{2}(1-1/p)+1/2}},\quad n-\frac{N}{2}(1-1/p)>1/2.
$

\item[(iii)]  $
\|  \frac{\mathcal{G}_{n,h}-\mathcal{G}_{n-1,h}}{h} \|_{p} \leq C_p\dfrac{1}{(nh)^{\frac{N}{2}(1-1/p)+1}},\quad n-\frac{N}{2}(1-1/p)>1.
$
\\
\end{itemize}
Here, $C_p$ is a constant independent of $h$ and $n$.
\end{theorem}
\begin{proof}
(i) It is well known (see \cite[p.334 (3.326)]{G}) that there exist $C_p$ (independent of $t$) such that $||G_t||_{p} = C_p\frac{1}{t^{\frac{N}{2}(1-\frac{1}{p})}},$ $\| \nabla G_t \|_{p} =  C_p\frac{1}{t^{\frac{N}{2}(1-\frac{1}{p})+1/2}}$ and $\| \frac{\partial}{\partial t}G_t \|_{p} \leq  C_p\frac{1}{t^{\frac{N}{2}(1-\frac{1}{p})+1}}$ for $t>0$.  Then by \eqref{Integralrepresentationheat} one gets
$$
\|\mathcal{G}_{n,h}\|_{p}\leq \frac{C_p}{h^n\Gamma(n)}\int_{0}^{\infty}e^{\frac{-t}{h}} t^{n-\frac{N}{2}(1-\frac{1}{p})-1}\,dt=C_p\frac{\Gamma(n-\frac{N}{2}(1-\frac{1}{p}))}{h^{\frac{1}{2}(1-\frac{1}{p})}\Gamma(n)}\leq \frac{C_p}{(nh)^{\frac{N}{2}(1-\frac{1}{p})}},
$$
where we have applied \eqref{GammaAsymp}.
The cases (ii) and (iii) follow in a similar way than (i).

\end{proof}

\section{Asymptotic $L^p-L^q$ decay}\label{4}

Let $h>0,$ and the time mesh $\N_0^h$. Now we consider the non-homogenous problem
\begin{equation}\label{maineq2}
\left\{
\begin{array}{lll}
\delta_{\text{left}}u(nh,x) = \Delta u(nh,x)+g(nh,x),\quad n\in\N,\,x\in \mathbb{R}^N, \vspace{2mm}\\
u(0,x) =f(x),
\end{array}
\right.
\end{equation}
where $u,f,g$ are functions defined on $\N_0^h\times \R^N,$ $\R^N$ and $\N^h\times \R^N$ respectively.

Formally, from \eqref{maineq2}, one gets \begin{equation}\label{SolNoHom}
u(nh,x)=(\mathcal{G}_{n,h}*f)(x)+h\sum_{j=1}^n(\mathcal{G}_{j,h}*g((n-j+1)h,\cdot))(x),\quad n\in \N, x\in\R^N.
\end{equation}
The expression \eqref{SolNoHom} gives a classical solution of \eqref{maineq2} on $L^p(\R^N)$ ($1\leq p\leq \infty$) whenever $f,g(nh,\cdot)\in L^p(\R^N),$ for $n\in\N.$ For convenience, we write the classical solution as $u(nh,x)=u_c(nh,x)+u_p(nh,x),$ where $$u_c(nh,x)=(\mathcal{G}_{n,h}*f)(x)$$ and $$u_p(nh,x)=h\sum_{j=1}^n(\mathcal{G}_{n-j+1,h}*g(jh,\cdot))(x).$$

Next, let us present a result about the  $L^p-L^q$ asymptotic decay for $u_c.$

\begin{theorem}\label{teorema4.1} Let $1\leq q\leq p\leq \infty.$ If $f\in L^q(\R^N)$, then the solution $u$ of \eqref{maineq2} satisfies

\begin{itemize}

\item[(i)]  $
\| u_c(nh) \|_{p} \leq C_p\dfrac{1}{(nh)^{\frac{N}{2}(1/q-1/p)}}\| f\|_{q}.
$

\item[(ii)]  $
\| \nabla u_c(nh) \|_{p} \leq C_p\dfrac{1}{(nh)^{\frac{N}{2}(1/q-1/p)+1/2}}\| f\|_{q}.
$

\item[(iii)]  $
\| \delta_{\text{left}} u_c(nh) \|_p \leq C_p\dfrac{1}{(nh)^{\frac{N}{2}(1/q-1/p)+1}}\| f\|_{q}.
$
\\
\end{itemize}
Here, $C_p$ is a constant independent of $h$ and $n$.
\end{theorem}
\begin{proof}
Take $r\geq 1$ such that $1+1/p=1/q+1/r,$ and applying Young's inequality we get $$\| u_c(nh) \|_{p} =\|  \mathcal{G}_{n,h}*f\|_{p} \leq \|  \mathcal{G}_{n,h}\|_{r}\| f\|_{q}.$$ From Theorem \ref{Gaussianestimates} (i) follows the case (i). The other cases are similar.
\end{proof}

Now, assuming certain conditions on the function $g$ we get an asymptotic decay for $u_p.$

\begin{theorem} Let $1\leq q\leq p\leq \infty,$ and $g(nh,\cdot)\in L^q(\R^N)$ with $$\|g(nh,\cdot)\|_q\lesssim\frac{1}{(nh)^{\gamma}},$$ for $\gamma>0.$
\begin{itemize}

\item[(i)]  If $\gamma,\frac{N}{2}(1/q-1/p)\neq 1$, then the solution $u$ of \eqref{maineq2} satisfies $$
\| u_p(nh) \|_{p} \lesssim \frac{(nh)^{1-\min\{1,\gamma\}}}{(nh)^{\min\{1,\frac{N}{2}(1/q-1/p)\}}}.
$$

\item[(ii)]  If $\gamma= 1$ and $\frac{N}{2}(1/q-1/p)\neq 1$, then the solution $u$ of \eqref{maineq2} satisfies $$
\| u_p(nh) \|_{p} \lesssim \frac{\log nh}{(nh)^{\min\{1,\frac{N}{2}(1/q-1/p)\}}}.
$$

\item[(iii)]  If $\gamma\neq 1$ and $\frac{N}{2}(1/q-1/p)= 1$, then the solution $u$ of \eqref{maineq2} satisfies $$
\| u_p(nh) \|_{p} \lesssim \frac{\log nh}{(nh)^{\min\{1,\gamma\}}}.
$$

\item[(iv)]  If $\gamma,\frac{N}{2}(1/q-1/p)= 1$, then the solution $u$ of \eqref{maineq2} satisfies $$
\| u_p(nh) \|_{p} \lesssim \frac{\log nh}{nh}.
$$

\end{itemize}
\end{theorem}
\begin{proof}
Let $r\geq 1$ such that $1+1/p=1/q+1/r.$ By Young's inequality and Theorem \ref{Gaussianestimates} (i) one gets \begin{eqnarray*}
\| u_p(nh) \|_{p}  &\leq& h\sum_{j=1}^n\|\mathcal{G}_{n-j+1,h}\|_r \|g(jh,\cdot))\|_q\lesssim  h\sum_{j=1}^n\frac{1}{((n-j+1)h)^{\frac{N}{2}(1/q-1/p)}}\frac{1}{(jh)^{\gamma}}\\
&=&h\biggl(\sum_{j=1}^{[n/2]}+ \sum_{j=[n/2]+1}^n\biggr)\frac{1}{((n-j+1)h)^{\frac{N}{2}(1/q-1/p)}}\frac{1}{(jh)^{\gamma}}=I_1+I_2.
\end{eqnarray*}

On one hand, for $1\leq j\leq [n/2]$ we have $n/2\leq n-j+1,$ which in turn implies that $$I_1\lesssim \frac{1}{(nh)^{\frac{N}{2}(1/q-1/p)}}\sum_{j=1}^{[n/2]}\frac{1}{(jh)^{\gamma}}\lesssim \frac{(nh)^{1-\min\{1,\gamma\}}}{(nh)^{\frac{N}{2}(1/q-1/p)}},\quad \gamma\neq 1,$$ and $I_1\lesssim  \frac{\log nh}{(nh)^{\frac{N}{2}(1/q-1/p)}}$ when $\gamma=1.$

On the other hand, $$I_2\lesssim \frac{1}{(nh)^{\gamma}}\sum_{j=1}^{n}\frac{1}{(jh)^{\frac{N}{2}(1/q-1/p)}}\lesssim \frac{(nh)^{1-\min\{1,\frac{N}{2}(1/q-1/p)\}}}{(nh)^{\gamma}},\quad \gamma\neq 1,$$ and $I_2\lesssim  \frac{\log nh}{(nh)^{\gamma}}$ when $\frac{N}{2}(1/q-1/p)=1.$

\end{proof}

\section{Large-time behaviour of solutions}\label{5}

In the following we study the asymptotic behavior of solution of \eqref{maineq2}, more precisely we will prove as the solution $u_c+u_p$ converges asymptotically to a lineal combination of the mass of the initial data $f$ and the mass of the non-homogeneity $g.$ Moreover, we will able to state the rate of the convergence. Along the section we will assume the following:

\begin{enumerate}
\item[(a)] $f\in L^1(\R^N)$.
\item[(b)] There exists $\gamma>\max\{1,\frac{N}{2}(1-1/p)\}$  such that
$$
\|g(jh,\cdot)\|_1 \lesssim \frac{1}{j^{\gamma}},\quad j\in \N.
$$
\end{enumerate}
Set also $$M_c =\int_{\R^N} f(x)\,dx, \quad M_p=\sum_{j=1}^{\infty}\int_{\R^N} g(jh,x)\,dx.$$

Taking into account the previous notation, we present the next theorem.

\begin{theorem}\label{Theorem5.1}
Let $1\leq p\leq \infty.$ Assume the conditions $(a)$-$(b),$ and suppose that $u$ is the classical solution of \eqref{maineq2}.

\begin{itemize}

 \item[(i)] Then $$(nh)^{\frac{N}{2}(1-\frac{1}{p})}\|u_c(nh)-M_c\mathcal{G}_{n,h}\|_{p}\to 0,\quad \mbox{as}\quad n\to\infty,$$ and $$(nh)^{\frac{N}{2}(1-\frac{1}{p})}\|u_p(nh)-hM_p\mathcal{G}_{n,h}\|_{p}\to 0,\quad \mbox{as}\quad n\to\infty.$$

 \item[(ii)] Suppose in addition that $|x|f\in L^1(\R),$ then $$ (nh)^{\frac{N}{2}(1-\frac{1}{p})}\|u_c(nh)-M_c\mathcal{G}_{n,h}\|_{p}\lesssim (nh)^{-1/2}.$$

\end{itemize}
\end{theorem}

\begin{proof}
We start proving the assertion $(ii)$. Since $f, |x|f\in L^1(\R^N),$
by  Decomposition Lemma \ref{decolemma} there exists $\phi\in L^1(\mathbb{R}^N;\mathbb{R}^N)$ such that
$$f=M_c\delta_0+\mbox{div}\, \phi$$
in the distributional sense, and
$$\|\phi\|_{1}\leq C\int_{\mathbb{R}^N}|x||f(x)|\,dx<\infty.$$
Then
\begin{align*}
u_c(nh,x)&=(\mathcal{G}_{n,h}\ast (M_c\delta_0+\mbox{div}\, \phi(\cdot)))(x)\\
&=M_c\mathcal{G}_{n,h}(x)+(\nabla\mathcal{G}_{n,h}\ast \phi)(x),
\end{align*}
which implies
\begin{equation*}
\|u_c(nh)-M_c\mathcal{G}_{n,h}\|_{p}\leq C\|\nabla\mathcal{G}_{n,h}\|_{p}\||x|f(x)\|_{1}
\leq K_{p,f}\dfrac{1}{(nh)^{\frac{N}{2}(1-1/p)+1/2}},
\end{equation*}
where we have used part $(ii)$ of Theorem \ref{Gaussianestimates}.
This implies
\begin{equation}\label{mainineq}
\|\mathcal{G}_{n,h}\ast f-M_c\mathcal{G}_{n,h}\|_{p}\leq
K_{p,f}\dfrac{1}{(nh)^{\frac{N}{2}(1-1/p)+\frac{1}{2}}}.
\end{equation}

To prove the first part of assertion $(i)$, we choose a sequence $(\eta_j)\subset C_0^{\infty}(\mathbb{R}^N)$ such that  $\int_{\mathbb{R}^N}\eta_j(x)\,dx=M_c$ for all $j,$ and $\eta_j\to f$ in $L^1(\mathbb{R}^N)$. For each $j$,  by Theorem \ref{Gaussianestimates} and \eqref{mainineq}  we get
\begin{align*}
\|u_c(nh)-M_c\mathcal{G}_{n,h}\|_{p}&=
\|\mathcal{G}_{nh}\ast f-M_c\mathcal{G}_{n,h}\|_{p}\\
&\leq \|\mathcal{G}_{n,h}\ast(f-\eta_j)\|_{p}+\|\mathcal{G}_{n,h}\ast \eta_j-M_c\mathcal{G}_{n,h}\|_{p}\\
&\leq \|\mathcal{G}_{n,h}\|_{p}\|f-\eta_j\|_{1}+\|\mathcal{G}_{n,h}\ast \eta_j-M_c\mathcal{G}_{n,h}\|_{p}\\
&\leq C_p\dfrac{1}{(nh)^{\frac{N}{2}(1-1/p)}}\|f-\eta_j\|_{1}+K_{p,\eta_j}\dfrac{1}{(nh)^{\frac{N}{2}(1-1/p)+\frac{1}{2}}}.
\end{align*}
It follows that
$$
(nh)^{\frac{N}{2}(1-1/p)}
\|u_c(nh)-M_c\mathcal{G}_{n,h}\|_{p}\leq
C_p\|f-\eta_j\|_{1}+K_{p,\eta_j}\frac{1}{(nh)^{\frac{1}{2}}},
$$
which implies
$$\limsup_{n\to\infty}\,(nh)^{\frac{N}{2}(1-1/p)}
\|u_c(nh)-M_c\mathcal{G}_{n,h}\|_{p}\leq
C_p\|f-\eta_j\|_{1}.$$
The assertion follows by letting $j\to\infty$.

Next, let us prove the second part of $(i)$. We can write

$$M_p=\sum_{j=1}^{n}\int_{\R^N} g(jh,x)+\sum_{j=n+1}^{\infty}\int_{\R^N} g(jh,x)\,dx.$$
It follows from Theorem \ref{Gaussianestimates} that
\begin{align*}
&(nh)^{\frac{N}{2}(1-\frac{1}{p})}\left\|\mathcal{G}_{n,h}(\cdot)\sum_{j=n+1}^{\infty}\int_{\R^N} g(jh,x)\,dx\right\|_{p}\\
&\leq (nh)^{\frac{N}{2}(1-\frac{1}{p})}\|\mathcal{G}_{n,h}(\cdot)\|_{p}
\sum_{j=n+1}^{\infty}\int_{\R^N} |g(jh,x)|\,dx\to0,\quad n\to\infty.
\end{align*}
Therefore it is enough to show the following
$$(nh)^{\frac{N}{2}(1-\frac{1}{p})}\left\|h\sum_{j=1}^{n}(\mathcal{G}_{n-j+1,h}\ast g(jh,\cdot))(\cdot)-h\mathcal{G}_{n,h}(\cdot)\sum_{j=1}^{n}\int_{\R^N} g(jh,y)\,dy\right\|_{p}\to0,\,\,\,n\to\infty.$$

Ir order to prove the assertion, we fix
$0<\delta<\frac{1}{10}.$
In particular, this implies that $0<\delta<\dfrac{1}{5}<\dfrac{1}{2}$ and
$$
\frac{\delta}{1-\delta}<\frac{1}{4}.
$$
Next, we decompose the set $\{1,2,3,..,n\}\times \mathbb{R}^N$ into two parts
$$\Omega_1:=\{1,2,...,\lceil n\delta \rceil\}\times \{y\in\mathbb{R}^N:|y|\leq (\delta nh)^{1/2}\},$$
$$\Omega_2:=\{1,2,...,n\}\times \mathbb{R}^N\setminus \Omega_1.$$

Let us start with the set $\Omega_1$.  By the integral form of the Minkowski inequality we get
\begin{align*}
&\left\|\sum_{(j,y)\in \Omega_1}\int(h \mathcal{G}_{n-j+1,h}(\cdot-y)-h \mathcal{G}_{n,h}(\cdot))g(jh,y)dy\right\|_{p}\\
&\leq h\sum_{(j,y)\in \Omega_1}\int \| \mathcal{G}_{n-j+1,h}(\cdot-y)- \mathcal{G}_{n,h}(\cdot)\|_{p}|g(jh,y)|dy.
\end{align*}

Note that in this set the following inequalities hold
\begin{equation}\label{inequalities}
n\geq n-j+1\geq n(1-\delta)>\frac{n}{2},
\end{equation}
where the second inequality follows from $\delta n-\lceil \delta n \rceil\geq -1.$ Now, when $(j,y)\in \Omega_1,$ we consider the following subsets over $\R^N$
$$A=\{x\in\R^N\,:\, |x-y|\leq 2(\delta nh)^{1/2}\},\qquad B:=\{x\in\R^n\,:\,|x-y|> 2(\delta nh)^{1/2}\},$$
and we write the $p$-norm over $\Omega_1$ in the following way \begin{eqnarray*}
\| \mathcal{G}_{n-j+1,h}(\cdot-y)- \mathcal{G}_{n,h}(\cdot)\|_{p}&\leq&\biggl(\int_A | \mathcal{G}_{n-j+1,h}(x-y)- \mathcal{G}_{n,h}(x)|^p \,dx \biggr)^{1/p}\\
&&+\biggl(\int_B | \mathcal{G}_{n-j+1,h}(x-y)- \mathcal{G}_{n,h}(x)|^p \,dx \biggr)^{1/p}.
\end{eqnarray*}

Let us estimate on $\Omega_1$ the part $A$ of the $p$-norm.  First we write \begin{align*}
\biggl( \int_A |\mathcal{G}_{n-j+1,h}(x-y)- \mathcal{G}_{n,h}(x)| \,dx\biggr)^{1/p}&\leq
\biggl(\int_A|  \mathcal{G}_{n-j+1,h}(x-y)|^p\,dx\biggr)^{1/p}\\
&+\biggl(\int_A| \mathcal{G}_{n,h}(x)|^p\,dx\biggr)^{1/p}\\
&=:I_1+I_2.
\end{align*}
For $(j,y)\in \Omega_1$  and $x\in A$ we have that
\begin{align*}
\frac{|x-y|^2}{(n-j+1)h}\leq \frac{4\delta}{(1-\delta)}<1.
\end{align*}
Since we want to estimate the solution for large values of $n$, we can assume that $n>N$. Thus, \eqref{inequalities} implies that $n-j+1>n/2>N/2.$
It follows from Theorem \ref{Pointwiseestimates} (i)
that
$$| \mathcal{G}_{n-j+1,h}(x-y)|\sim \frac{1}{((n-j+1)h)^{N/2}}.$$
Then
\begin{align*}
(nh)^{\frac{N}{2}(1-\frac{1}{p})} h I_1&\lesssim \frac{(nh)^{\frac{N}{2}(1-\frac{1}{p})}h}{((n-j+1)h)^{N/2}}\biggl(\int_A \, dx\biggr)^{1/p}\\
&=\frac{C_p (nh)^{\frac{N}{2}(1-\frac{1}{p})} h (\delta nh)^{N/2p}}{((n-j+1)h)^{N/2}}\leq C_p h \delta^{N/2p},
\end{align*}
where in the last inequality we have used \eqref{inequalities}.  Analogously, for  $(j,y)\in \Omega_1$ and $x\in A$ we have
$$\frac{|x|^2}{nh}\leq \frac{(|x-y|+|y|)^2}{nh}\leq 9\delta<1,$$
which implies that
$$|\mathcal{G}_{n,h}(x)|\sim \frac{1}{(nh)^{N/2}}.$$
Therefore
\begin{align*}
(nh)^{\frac{N}{2}(1-\frac{1}{p})}hI_2\lesssim C_p h \delta^{N/2p}.
\end{align*}
Since $\sum_{j=1}^{\infty}\|g(jh,\cdot)\|<\infty,$ we get
$$
(nh)^{N/2(1-1/p)}h\sum_{(j,y)\in \Omega_1}\int \biggl(\int_A | \mathcal{G}_{n-j+1,h}(x-y)- \mathcal{G}_{n,h}(x)|^p \,dx \biggr)^{1/p}|g(jh,y)|\,dy\leq C_p h\delta^{N/2p}\to0,\quad \delta\to0.
$$

Now we consider on $\Omega_1$ the part $B$ of the $p$-norm. We write \begin{align*}
\biggl( \int_B |\mathcal{G}_{n-j+1,h}(x-y)- \mathcal{G}_{n,h}(x)|^p \,dx\biggr)^{1/p}&\leq
\biggl( \int_B |  \mathcal{G}_{n-j+1,h}(x-y)- \mathcal{G}_{n-j+1,h}(x)|^p\,dx\biggr)^{1/p}\\
&+\biggl( \int_B| \mathcal{G}_{n-j+1,h}(x)- \mathcal{G}_{n,h}(x)|^p\,dx\biggr)^{1/p}\\
&=:I_3+I_4.
\end{align*}
First, let us estimate $I_3$. By mean value theorem there exists $\tilde{x}$ between $x-y$ and $x$ ($x$ denote the integration variable) such that
\begin{equation*}\label{I3}I_3= |y| \biggl( \int_B |  \nabla\mathcal{G}_{n-j+1,h}(\tilde{x})|^p\,dx\biggr)^{1/p}.\end{equation*}
Since $|y|\leq (\delta nh)^{1/2}<\frac{1}{2}|x-y|$ then
\begin{equation}\label{ine}|\tilde{x}|\geq|x-y|-|\tilde{x}-(x-y)|\geq |x-y|-|y|\geq \frac{|x-y|}{2},\end{equation}
and \begin{equation}\label{ine2}|\tilde{x}|\leq |x-y|+|y|\leq|x-y|+\frac{|x-y|}{2}=\frac{3}{2}|x-y|.\end{equation} Equations \eqref{ine} and \eqref{ine2} show that $|\tilde{x}|$ and $|x-y|$ are comparable. Also, by \eqref{inequalities} and \eqref{ine} we obtain
\begin{equation}\label{tildex}
\frac{|\tilde{x}|}{((n-j+1)h)^{1/2}}\geq \frac{|x-y|}{2((n-j+1)h)^{1/2}}> \delta^{1/2}.
\end{equation}
Now we will use the asymptotics of $|  \nabla\mathcal{G}_{n-j+1,h}(\tilde{x})|,$ so we divide $I_3$ in two parts, $I_{31}$ and $I_{32}$ depending on whether $\frac{|\tilde{x}|}{((n-j+1)h)^{1/2}}$ is less or greater than 1 respectively (we are assuming $\delta$ enough small).

In $I_{32}$, when $\frac{|\tilde{x}|}{((n-j+1)h)^{1/2}}\geq1$, by \eqref{ine2} one gets
$$((n-j+1)h)^{1/2}\leq|\tilde{x}|\leq \frac{3}{2}|x-y|.$$
For this reason, the integration region in $I_{32}$ is contained in $\{x\in\R^N\,:\,\frac{2}{3}((n-j+1)h)^{1/2}\leq |x-y|\}$. From Theorem \ref{Pointwiseestimates} (ii), the fact that $|y|\leq (\delta nh)^{1/2},$ \eqref{inequalities} and \eqref{ine}, we have
\begin{align*}
I_{32}&\leq C(\delta nh)^{1/2}\left(\int_{|x-y|
\geq \frac{2}{3}((n-j+1)h)^{1/2}}\left(\frac{(n-j+1)h}{|\tilde{x}|^{N+3}}\right)^p dx\right)^{1/p}\\
&\leq C(\delta nh)^{1/2}\left(\int_{|x-y|
\geq \frac{2}{3}((n-j+1)h)^{1/2}}\left(\frac{nh}{|x-y|^{N+3}}\right)^p dx\right)^{1/p}\\
&= C_p\delta ^{1/2}(nh)^{3/2}((n-j+1)h)^{N/(2p)-(N+3)/2}\\
&\leq C_p\delta^{1/2}(nh)^{-N/2(1-1/p)}.
\end{align*}
Consequently,
$$(nh)^{N/2(1-1/p)}hI_{32}\leq C_ph\delta^{1/2}.$$

For $I_{31},$ by \eqref{tildex} note that the set of integration is contained in $\{x\in\R^N\,:\, 1\geq \frac{|x-y|}{2((n-j+1)h)^{1/2}}> \delta^{1/2} \}.$ Then from Theorem \ref{Pointwiseestimates} (ii) it follows
\begin{align*}
I_{31}&\leq C(\delta nh)^{1/2}\left(\int_{\delta^{1/2}
\leq \frac{|x-y|}{((n-j+1)h)^{1/2}}\leq 1}\left(\frac{|\tilde{x}|}{((n-j+1)h)^{N/2+1}}\right)^p\,dx\right)^{1/p}\\
&\leq C \frac{\delta^{1/2}(nh)^{1/2}}{(nh)^{N/2+1}}\left(\int_{\delta^{1/2}
\leq \frac{|x-y|}{((n-j+1)h)^{1/2}}\leq1}|x-y|^p\,dx\right)^{1/p}\\
&\leq C_p \delta^{1/2}(nh)^{-N/2(1-1/p)}(1-\delta^{(N+p)/2})^{1/p}\\
&\leq C_p \delta^{1/2}(nh)^{-N/2(1-1/p)},
\end{align*}
which is equivalent to
$$(nh)^{N/2(1-1/p)}hI_{31}\leq C_ph\delta^{1/2}.$$

Next, let us estimate $I_4$.  From discrete mean value theorem (see \cite[Corollary 2]{AP}), there exist $\tilde{n}\in\{n-j+2,...,n\}$ (whenever $j\geq 2$) and $C>0$ such that
\begin{equation}\label{I4}
I_4\leq C(j-1)h\left( \int_{B}|\mathcal{G}_{\tilde{n},h}(x)-\mathcal{G}_{\tilde{n}-1,h}(x)|\,dx\right)^p=C(j-1)h\left( \int_{B}|\Delta \mathcal{G}_{\tilde{n},h}(x)|\,dx\right)^p.
\end{equation}
Recall that in $\Omega_1$ we have $n-j+1\leq \tilde{n}\leq n,$ which implies by \eqref{inequalities} that $nh(1-\delta)\leq \tilde{n}h\leq nh.$ Also, in $\Omega_1$ and $B$ we have
$$
|x|=|x+y-y|\geq |x-y|-|y|\geq (\delta nh)^{1/2},
$$
so
$$
\tilde{z}:=\frac{|x|}{(\tilde{n}h)^{1/2}}\geq \frac{|x|}{(nh)^{1/2}}\geq \frac{(\delta nh)^{1/2}}{(nh)^{1/2}}=\delta^{1/2},
$$
and we have again two cases. We denote by $I_{41}$ and $I_{42}$ depending on whether $\tilde{z}\leq 1$ or $\tilde{z}\geq 1$ on the right side of \eqref{I4}.

For $I_{41},$ since $\tilde{z}\leq 1$ and $|x|\geq (\delta nh)^{1/2}$ the set of integration is contained in $\{ x\in\R^N\,:\, (\delta nh)^{1/2}\leq |x|\leq (nh)^{1/2}\}.$ Then, from Theorem \ref{Pointwiseestimates} (iii) and the fact that we are in $\Omega_1$, we have
\begin{align*}
I_{41}&\leq C(j-1)h\left(\int_{(\delta nh)^{1/2}
\leq |x|\leq (nh)^{1/2}}\frac{1}{(\tilde{n}h)^{(N/2+1)p}}\,dx\right)^{1/p}\\
&\leq \frac{C\delta nh}{(nh(1-\delta))^{N/2+1}}\left(\int_{(\delta nh)^{1/2}
\leq |x|\leq (nh)^{1/2}}dx\right)^{1/p}\\
&= \frac{C_p\delta(1-\delta^{N/2})^{1/p}(nh)^{-N/2(1-1/p)}}{(1-\delta)^{N/2+1}}.
\end{align*}
Consequently,
$$
(nh)^{N/2(1-1/p)}hI_{41}\leq \frac{C_p h\delta}{(1-\delta)^{N/2+1}}.
$$

For $I_{42}$ we have
$$
1\leq \tilde{z}=\frac{|x|}{(\tilde{n}h)^{1/2}}\leq \frac{|x|}{(nh(1-\delta))^{1/2}},
$$
which implies that the set of integration is contained in $\{ x\in\R^N\,:\, |x|\geq (nh(1-\delta))^{1/2}\}.$ Then
\begin{align*}
I_{42}&\leq  C(j-1)h \left(\int_{|x|\geq((1-\delta)nh)^{1/2}}\frac{1}{|x|^{(N+2)p}}\,dx\right)^{1/p}\\
&= \frac{C_p\delta (nh)^{-N/2(1-1/p)}}{(1-\delta)^{N/2(1-1/p)+1}}.
\end{align*}
Consequently,
$$
(nh)^{1/2(1-1/p)}hI_{42}\leq \frac{C_p h\delta}{(1-\delta)^{N/2(1-1/p)+1}}.
$$

Collecting all above terms over $B$ we get
\begin{align*}
&(nh)^{N/2(1-1/p)}h\sum_{(j,y)\in \Omega_1}\int \biggl( \int_B |\mathcal{G}_{n-j+1,h}(x-y)- \mathcal{G}_{n,h}(x)|^p \,dx\biggr)^{1/p} |g(jh,y)|dy\\
\nonumber &\leq C_p\delta^{\eta}\sum_{j=1}^n\int_{\mathbb{R}^N}|g(jh,y)|dy
\end{align*}
for some positive number $\eta$. The upper bound tends to zero as $\delta\to 0$ uniformly in $nh$.

Now, we consider the set $\Omega_2$.  Then

\begin{align*}
&(nh)^{N/2(1-1/p)}h\sum_{(j,y)\in \Omega_2}\int \|\mathcal{G}_{n-j+1,h}(\cdot-y)-\mathcal{G}_{n,h}(\cdot)\|_p|g(jh,y)|dy\\
&\leq
(nh)^{N/2(1-1/p)}h\sum_{(j,y)\in \Omega_2}\int \|\mathcal{G}_{n-j+1,h}(\cdot-y)\|_p|g(jh,y)|dy\\
&+(nh)^{N/2(1-1/p)}h\sum_{(j,y)\in \Omega_2}\int \|\mathcal{G}_{n,h}(\cdot)\|_p|g(jh,y)|dy\\
&=:I_5+I_6.
\end{align*}
By Theorem \ref{Gaussianestimates} (i) one gets
$$
I_6\leq C_p \sum_{(j,y)\in \Omega_2}\int|g(jh,y)|dy.
$$
As $n\to\infty$, $\Omega_1\to \mathbb{N}\times \mathbb{R}^N$. This implies that $\Omega_2$ has measure zero, and since $\sum_{j=1}^{\infty}\int_{\mathbb{R}^N}|g(jh,y)|dy<\infty,$ then $\sum_{(j,y)\in \Omega_2}\int|g(jh,y)|dy\to 0$ as $n\to\infty$.  It follows that $I_6\to 0$ as $n\to\infty$.

For $I_5$ we have two possibilities: either $j\leq \lceil \delta n\rceil$ or $j> \lceil \delta n\rceil$. Thus, we divide
$$
\Omega_2=\{1,...,\lceil \delta n\rceil\}\times \{y\in \R:|y|>(\delta nh)^{1/2}\}
\cup\{\lceil \delta n\rceil+1,...,n\}\times \mathbb{R}^N.
$$
Then
\begin{align*}
I_5&\leq (nh)^{N/2(1-1/p)}h\sum_{j=1}^{\lceil \delta n\rceil}\int_{|y|>(\delta nh)^{1/2}} \|\mathcal{G}_{n-j+1,h}(\cdot-y)\|_p |g(jh,y)|dy\\
&+(nh)^{N/2(1-1/p)}h\sum_{j=\lceil \delta n\rceil+1}^{n}\int_{\mathbb{R}^N} \|\mathcal{G}_{n-j+1,h}(\cdot-y)\|_p |g(jh,y)|dy\\
&=:I_{51}+I_{52}.
\end{align*}
Let us start with $I_{51}$.  Recall that for $j\in \{1,...,\lceil \delta n\rceil\}$ the expression \eqref{inequalities} holds. So, by Theorem \ref{Gaussianestimates} (i) we have that
\begin{align*}
I_{51}&\leq C_p h \sum_{j=1}^{\lceil \delta n\rceil}\int_{|y|>(\delta nh)^{1/2}}|g(jh,y)|dy\to 0
\end{align*}
as $n\to\infty.$

Next, for $I_{52},$ again by Theorem \ref{Gaussianestimates} (i), \eqref{inequalities} and the fact that $\gamma>1$ we obtain
\begin{align*}
I_{52}&\leq C_p(nh)^{N/2(1-1/p)}h\sum_{j=\lceil \delta n\rceil+1}^{n}
\dfrac{1}{((n-j+1)h)^{\frac{N}{2}(1-1/p)}}\frac{1}{j^{\gamma}}.
\end{align*}
Thus, if $\frac{N}{2}(1-1/p)\in[0,1)$ and $\gamma>1,$ then $$I_{52}\leq \frac{C_p(nh)^{N/2(1-1/p)}h}{(\lceil \delta n\rceil+1)^{\gamma}h^{\gamma}}\sum_{j=1}^{n-\lceil \delta n\rceil}\frac{1}{(jh)^{\frac{N}{2}(1-1/p)}}\leq \frac{C_ph (nh)^{N/2(1-1/p)}(nh-\lceil \delta n\rceil h)^{1-N/2(1-1/p)}}{(\delta n h)^{\gamma}}\to 0,$$ as $n\to\infty.$ Also, if $\gamma>\frac{N}{2}(1-1/p)>1$,  then $$I_{52}\leq \frac{C_p(nh)^{N/2(1-1/p)}h}{(\lceil \delta n\rceil+1)^{\gamma}h^{\gamma}}\sum_{j=1}^{n-\lceil \delta n\rceil}\frac{1}{(jh)^{\frac{N}{2}(1-1/p)}}\leq \frac{C_ph (nh)^{N/2(1-1/p)}}{(\delta n h)^{\gamma}}\to 0,\quad n\to\infty.$$ The case $\gamma>\frac{N}{2}(1-1/p)=1$ implies, similarly to the previous one, that  $$I_{52}\leq \frac{C_ph (nh)\log(n)}{(\delta n h)^{\gamma}}\to 0,\quad n\to\infty.$$

\end{proof}

\section{Optimal $L^2$-decay for solutions}\label{6}
In this section we prove that the decay rate of the solution $u_c$ of \eqref{maineq1} given in Theorem \ref{teorema4.1} (i) is optimal.
\begin{theorem}
Let $u_c$ be the solution of \eqref{maineq1}.  Assume that $f\in L^1(\R^N)\cap L^2(\R^N)$ and $\int_{\R^N}f(x)dx\neq 0$. Then
$$
\|u_c(nh,\cdot)\|_2 \sim \frac{C}{(nh)^{N/4}},\quad nh\geq 1.
$$
\end{theorem}

\begin{proof}
Let $\rho>0,$ we have by Proposition \ref{HeatProperties} (iii) that
\begin{align}\label{5.1}
\nonumber\|u_c(nh,\cdot)\|_2^2&=\|\mathcal{F}u_c(n,\cdot)\|_2^2=\int_{\R^N}|\mathcal{F}\mathcal{G}_{n,h}(\xi)|^{2}\,\,|\mathcal{F}f(\xi)|^2\,d\,\xi\\
\nonumber&\geq \int_{\mathcal{B}(0,\rho)}\frac{1}{(1+h|\xi|^{2})^{2n}}|\mathcal{F}f(\xi)|^2d\,\xi\\
&\geq \frac{1}{(1+h|\rho|^{2})^{2n}}\int_{\mathcal{B}(0,\rho)}|\mathcal{F}f(\xi)|^2d\,\xi.
\end{align}
By Plancherel Theorem and the Riemann-Lebesgue Lemma we have that $\mathcal{F}f\in C_0(\R^N)\cap L^2(\R^N)$. By the Lebesgue differentiation theorem, we may choose $\rho_0$ small enough such that
$$
\frac{1}{\rho^N}\int_{\mathcal{B}(0,\rho)}|\mathcal{F}f(\xi)|^2d\,\xi\geq \frac{1}{2}\left|\mathcal{F}(0)\right|^2\quad
\mbox{for all}\quad \rho\in(0,\rho_0].
$$

Substituting the previous inequality in \eqref{5.1} we have that for all $\rho\in(0,\rho_0]$
$$
\|u_c(nh,\cdot)\|_2^2\geq \frac{\rho^N}{2(1+h|\rho|^{2})^{2n}}\left|\mathcal{F}(0)\right|^2.
$$
We choose $\rho:=\dfrac{\rho_0}{(nh)^{1/2}}$. For $n$ enough large, $nh\geq 1$ then $\rho$ belongs to $(0,\rho_0)$.  Hence
\begin{align*}
\frac{\rho^N}{(1+h|\rho|^{2})^{2n}}&=\frac{\rho_0^N}{(nh)^{N/2}\left(1+\dfrac{\rho_0^2}{n}\right)^{2n}}\\
&\geq \frac{\rho_0^N}{(nh)^{N/2}e^{2\rho_0^2}}=\frac{C}{(nh)^{N/2}},\quad nh\geq 1,
\end{align*}
and then we get the first assertion of the result.

Next, let us prove the upper bound.  By Plancherel's Theorem and the Riemann-Lebesgue Lemma  we have
\begin{align*}
\|u_c(nh,\cdot)\|_2^2&=\int_{\R^N}\frac{1}{(1+h|\xi|^{2})^{2n}}|\mathcal{F}f(\xi)|^2d\,\xi\\
&\leq \|\mathcal{F}f\|_{\infty}^2\int_{\R^N}\frac{1}{(1+h|\xi|^{2})^{2n}}d\,\xi\leq \|f\|_{1}^2\int_{\R^N}\frac{1}{(1+nh|\xi|^{2})^2}d\,\xi\\
&= \frac{C \|f\|_{1}^2}{(nh)^{N/2}}\int_{\R^N}\frac{1}{(1+|\xi|^{2})^2}d\,\xi= \frac{C}{(nh)^{N/2}}.
\end{align*}

\end{proof}

\section{Appendix}\label{Appendix}\label{7}

Here, we present some useful facts which are needed in order to obtain our results.\\

First, we recall the following asymptotic behavior of the Gamma function.  Let $\alpha, z\in\C$, then
\begin{equation}\label{GammaAsymp}
\frac{\Gamma(z+\alpha)}{\Gamma(z)}=z^{\alpha}\biggl(   1+\frac{\alpha(\alpha+1)}{2z}+O(|z|^{-2}) \biggr),\quad |z|\to\infty,\end{equation}
whenever $z\neq 0,-1,-2,\ldots,$ and $z\neq -\alpha,-\alpha-1,\ldots,$ see \cite{ET}.

Next, we recall the definition of Bessel functions and some basic results which are used in this work. See \cite{G,Mc,W} for more information about this topic.

Let $\nu \in \R$. The \emph{Modified Bessel functions of the first kind} are defined by
\begin{equation*}\label{IBessel}
I_{\nu}(x)=\sum_{n=0}^{\infty}\frac{1}{\Gamma(n+\nu+1)n!}\left(\frac{x}{2}\right)^{2n+\nu}.
\end{equation*}
Such functions allow to define, for $\nu \in \R$ a non entire number, the \emph{Modified Bessel functions of second kind} or \emph{MacDonald's functions}  as follows \begin{equation*}\label{KBesselnoentire}
	K_{\nu}(x)=\frac{\pi}{2}\frac{I_\nu(x)-I_{-\nu}(x)}{\sin(\nu x)}.
	\end{equation*}
For the case $\mu \in \Z$ they are defined by
	\begin{equation*}\label{KBesselentire}
	K_{\mu}(x)=\lim_{\nu\to\mu}K_{\nu}(x)=\lim_{\nu\to\mu}\frac{\pi}{2}\frac{I_\nu(x)-I_{-\nu}(x)}{\sin(\nu x)}.
	\end{equation*}

These functions arise as the solutions for the ODE
\begin{equation*}\label{KBesselODE}
\frac{d^2}{dz^2}u(z)=\left(1+\frac{\nu^2}{z^2}\right)u(z)-\frac{1}{z}\frac{d}{dz}u(z).
\end{equation*}
Some properties of the MacDonald's functions used along the paper are the following ones:
\begin{enumerate}
	\item[(P1)] $K_{\nu}(z)=\int_{0}^{\infty}e^{-z\cosh t}\cosh(\nu t)\ dt, \quad |arg(z)| < \frac{\pi}{2}\ \mbox{or}\ \Re z = 0 \ \mbox{and } \nu=0$.
	\item[(P2)] $z\frac{d}{dz}K_{\nu}(z)+\nu K_{\nu}(z)=-zK_{\nu-1}(z)$.
	\item[(P3)] $z\frac{d}{dz}K_{\nu}(z)-\nu K_{\nu}(z)=-zK_{\nu+1}(z)$.
	\item[(P4)]When $0<z\ll \sqrt{\nu+1}$, we have
	\begin{equation*}\label{MacAsymp}
	K_{\nu}(z)\sim \frac{\Gamma(\nu)}{2}\left(\frac{2}{z}\right)^{\nu}, \quad \mbox{if} \ \nu\neq 0.
	\end{equation*}
	\item[(P5)]$K_{\nu}(z)=\biggl( \dfrac{\pi}{2z} \biggr)^{1/2}e^{-z}\biggl(1+O(1/z) \biggr),\quad z\to\infty.$
	\item[(P6)]$K_{\nu}=K_{-\nu}.$
	\item[(P7)]$zK_{\nu-1}(z)-zK_{\nu+1}(z)=-2\nu K{\nu}(z).$
\end{enumerate}

We also need in this paper the following decomposition lemma (see \cite{DZ}).

\begin{lemma}\label{decolemma}
Suppose $f\in L^1(\mathbb{R}^N)$ such that $\int_{\mathbb{R}^N}|x||f(x)|dx<\infty.$  Then there exists $F\in L^1(\mathbb{R}^N;\mathbb{R}^N)$ such that
\begin{equation*}
f=\left(\int_{\mathbb{R}^N}f(x)dx\right)\delta_0+\mbox{div}\, F
\end{equation*}
in the distributional sense and
\begin{equation*}
\|F\|_{L^1(\mathbb{R}^N;\mathbb{R}^N)}\leq C_d\int_{\mathbb{R}^N}|x||f(x)|dx.
\end{equation*}
\end{lemma}

\noindent {\it Acknowledgments.} The authors would like to thank to Jorge Gonz\'alez-Camus by his help and advice with the pictures along the paper.

\end{document}